\documentclass[12pt]{article}
\usepackage{mathrsfs}
\usepackage{amssymb}
\usepackage{color}
\usepackage{amsmath}
\allowdisplaybreaks[4]

\usepackage{amsthm}
\usepackage{amsmath}
\usepackage{graphicx}

 \newtheorem{thm}{Theorem}[section]
 
 \newtheorem{lem}[thm]{Lemma}
 
 \theoremstyle{definition}

 \numberwithin{equation}{section}

\newtheorem{theorem}{Theorem}[section]
\newtheorem{lemma}[theorem]{Lemma}

\theoremstyle{definition}

\theoremstyle{remark}
\newtheorem{remark}[theorem]{Remark}

\begin{document}

\title{Global Well-posedness for 2D Nonlinear Wave Equations without Compact Support}

\author{Yuan Cai\footnote{School of Mathematical Sciences, Fudan University, Shanghai 200433, P. R. China. Email: ycai14@fudan.edu.cn}
\and
Zhen Lei\footnote{School of Mathematical Sciences, Fudan University; Shanghai Center for Mathematical Sciences, Shanghai 200433, P. R. China. Email: zlei@fudan.edu.cn}
\and
Nader Masmoudi
\footnote{Courant Institute of Mathematical Sciences, New York University, NY 10012, USA.
Email: masmoudi@cims.nyu.edu}
}
\date{}
\maketitle

\begin{abstract}
In the significant work of \cite{Alinhac01a}, Alinhac proved the global existence of small solutions for 2D quasilinear wave equations under the null conditions.
The proof heavily relies on the fact that the initial data have compact support \cite{LSZ13}. Whether this constraint can be removed or not is still unclear. In this paper, for fully nonlinear wave equations under the null conditions, we prove the global well-posedness for small initial data without compact support. Moreover, we apply our result to a class of quasilinear wave equations.
\end{abstract}

%
%
\maketitle





\section{Introduction}

Global well-posedness for nonlinear wave equations is a well-oiled mathematical topic.
Many mathematicians including S. Alinhac, D. Christodoulou, L. H\"{o}rmander, F. John, S. Klainerman, etc.  have made tremendous contributions to this subject.
The first nontrivial long-time existence result was established by John and Klainerman in \cite{JohnKlainerman84} where they showed the almost global existence for 3D quasilinear scalar wave equations.
The landmark work of global existence for 3D quasilinear wave equations was obtained firstly by Klainerman \cite{Klainerman86} and by Christodoulou \cite{Christodoulou86} independently under null conditions.
The corresponding problem in 2D is more delicate since it only has critical decay even under the null condition.
In 2001, Alinhac \cite{Alinhac01a} introduced the so-called ghost weight and proved the global well-posedness.
However, the argument in \cite{Alinhac01a} heavily relies on the compact support of the initial data, since the use of a certain kind of Hardy-type inequality depending on the compact support is crucial \cite{LSZ13}. For the case where the initial data is not compactly supported, it's still unclear whether the existence of solution is global or not.

In this paper, we consider the following 2D fully nonlinear wave equations:
\begin{equation}\label{wave-equation}
\begin{cases}
\Box u=N_{\alpha\beta\mu\nu} \partial_\alpha \partial_\beta u \partial_{\mu} \partial_{\nu}u, \\
u(0,\cdot)=\varphi, \partial_t u(0,\cdot)=\psi,
\end{cases}
\end{equation}
where $\partial:=(\partial_t, \partial_1, \partial_2)$, $\Box$ is d' Alembertian operator, $\varphi, \psi\in H^k_{\Lambda}$ (The space $H^k_{\Lambda}$ will be explained in Section 2), $u=u(t,x_{1},x_{2})$ is the unknown variable.
Here and throughout this paper, Einstein's summation convention is used, which means that repeated indices are always summed over their ranges.
Our first goal is to prove the global well-posedness for \eqref{wave-equation} without compact support.

We denote the nonlinearities as
\begin{equation} \label{nonlinearities}
N(u,v)=N_{\alpha\beta\mu\nu} \partial_\alpha \partial_\beta u \partial_{\mu} \partial_{\nu}u.
\end{equation}
It's clear that, by a simple symmetrization procedure, $N_{\alpha\beta\mu\nu}$ can be assumed to be satisfying
the symmetry
\begin{equation} \label{symmetry-condition}
N_{\alpha\beta\mu\nu}=N_{\beta\alpha\mu\nu}=N_{\alpha\beta\nu\mu}.
\end{equation}
For global existence, we impose the standard null condition for $N(u,v)$:
\begin{equation}\label{null-condition}
N_{\alpha\beta\mu\nu}X_\alpha X_\beta X_\mu X_\nu=0
\end{equation}
for all $X\in \Sigma$, where $\Sigma$ is the light cone
$$
\Sigma=\{X\in \mathbb{R}^3:X_0^2=X_1^2 + X_2^2\}.
$$

The main result of this paper can be described as follows:

\begin{thm}\label{thm}
Let $M> 0$, $0<\gamma< \frac18$ be two given constants and $(\varphi,\psi)\in H^k_\Lambda$, with $k\geq 8$.
Suppose that the nonlinearities satisfy the symmetry \eqref{symmetry-condition}, null condition \eqref{null-condition},
and
\begin{equation}\label{thm-InitialDataCondition}
\|(\varphi,\psi)\|_{ H^k_\Lambda}<M,\quad \|(\varphi,\psi)\|_{ H^{k-1}_\Lambda}<\epsilon.
\end{equation}
There exists
a positive constant $\epsilon_0<e^{-M}$ which depend on $M,\ k,\ \gamma$
such that, if $\epsilon \leq \epsilon_0$, the fully
nonlinear wave equation \eqref{wave-equation} with initial data
$(u(0), \partial_tu(0))=(\varphi,\psi)$ has a unique global solution
which satisfies
$E_k(t)\leq C_0 M^2 \langle t\rangle ^{\gamma}$ and $E_{k-1}(t)\leq C_0\epsilon^2e^{C_0M}$
for some $C_0>1$ uniformly for $0\le t<\infty$.
\end{thm}

\begin{remark}
We emphasize that our argument can also be applied to the general fully nonlinear wave equations
with cubic and higher-order nonlinear terms whose quadratic and cubic terms both satisfy null conditions.
 For simplicity, we only focus on the quadratic case here, since the higher-order terms can be treated similarly, see \cite{Alinhac01a,Sideris00}.
\end{remark}

Our second goal is to apply Theorem \ref{thm} to show the global
well-posedness for a class of quasilinear wave equations
even when the initial data is not compactly supported.
There are two key points for this purpose. The first one is
to transform the quasi-linear wave equations
to fully nonlinear equations so that the application of Theorem \ref{thm} is possible.
In general the resulting fully nonlinear wave equations
do not satisfy the symmetry condition \eqref{symmetry-condition}, but do satisfy certain null
condition if the original quasi-linear wave equations do. To avoid the loss of derivatives, we need to symmetrize
the resulting fully nonlinear wave equations and at the mean time, keep the null structure.
The second point is to show that the null condition
\eqref{null-condition} is preserved under the symmetrization procedure. See Lemma \ref{lemmaSymmetry} for more details.

Consider the following quasilinear wave equations:
\begin{equation} \label{quasi-equation}
\begin{cases}
\Box v=A_l\partial_l(N_{\mu\delta}\partial_\mu v\partial_\delta v), \\
v(0,x)=v_0(x),\partial_t v(0,x) =v_1(x).
\end{cases}
\end{equation}
Now we are ready to state the second
main result of this paper as follows:
\begin{thm} \label{thm-quasi}
For the quasilinear wave equations \eqref{quasi-equation}, we assume the following null condition holds
\begin{equation} \label{quasi-null}
 N_{\mu\delta} X_\mu X_\delta=0
\end{equation}
for all $X\in \Sigma$.
Let  $(\varphi,\psi)$ defined by \eqref{InitialData1} or by \eqref{InitialData2}-\eqref{InitialDataCondition3} in Section 5 belong to $H^k_\Lambda$ with  $k\geq 8$ and
they satisfy the condition \eqref{thm-InitialDataCondition}.
Then the equations \eqref{quasi-equation} with initial data satisfying
\eqref{InitialData1} or \eqref{InitialData2}-\eqref{InitialDataCondition3}
have global classical solutions.
\end{thm}
\begin{remark}
We mention that \eqref{quasi-null} is equivalent to
\begin{equation} \label{quasi-null2}
A_l N_{\mu\delta}X_l  X_\mu X_\delta=0.
\end{equation}
See Section 5 for more details.
\end{remark}
As an example of \eqref{quasi-equation}, we consider the following prototype quasilinear wave equation
\begin{equation} \label{quasi-equation-standard}
\begin{cases}
\Box v=\partial_t ( |\partial_t v|^2 - |\nabla v|^2 ), \\
v(0,\cdot)=v_0,  \partial_t v(0,\cdot)=v_1.
\end{cases}
\end{equation}
In this case, $\varphi$ and $\psi$ which appear in Theorem \ref{thm-quasi} are
\begin{equation}
\varphi=\chi,\quad
\psi=v_0,
\end{equation}
where $\chi$ is a function satisfying the following elliptic equation
\begin{equation}
\begin{cases}
-\Delta\chi=|v_1|^2-|\nabla v_0|^2-v_1,\\
(\chi(x),\nabla\chi(x)) \in H^k_\Lambda.
\end{cases}
\end{equation}
\begin{remark}\label{RemarkUniformBound}
Indeed, we can obtain the uniform bound of the highest-order energy for the equation \eqref{quasi-equation-standard} (See Section 6). In \cite{AlinhacBook}, Alinhac proved that for the three dimensional scalar quasilinear wave equation with
small initial data under null condition, the highest-order energy is uniformly bounded (see also Wang \cite{Wang14}), and he
also conjectured that a certain time growth of the highest-order energy is a true phenomenon except for 3D scalar wave equation. However, in \cite{LeiWang15}, Lei-Wang were able to show that the uniform boundedness of the highest-order energy is still true
for 3D incompressible elastodynamics. Here we provide another counter-example to Alinhac's conjecture. We emphasize that it is still unclear whether Alinhac's conjecture is true or not for general 2D quasilinear wave equations, 3D non-relativistic wave equations and compressible elastodynamics.
\end{remark}

This paper is mainly inspired by the recent work of global existence result for 2D incompressible elastodynamics under a kind of \textit{strong null condition} in \cite{Lei14}.
In the case of general 2D quasilinear wave equations, to prove the global existence of small solutions under null conditions, Alinhac \cite{Alinhac01a} used a kind of
Hardy-type inequality to produce good unknowns $\partial(\partial_t+\partial_r)\Gamma^\alpha u$  which decay as $\langle t \rangle ^{-1}$ in $L^2$.
For that purpose,
 the initial data is required to have compact support.
Our strategy is to focus on the fully nonlinear cases in which the equations naturally have one more derivative
than the quasilinear ones. Hence we can gain an extra $\langle t\rangle^{-1}$ decay in the lower-order energy estimate.
Since we don't need to use Hardy-type inequality to create an extra derivative, the constraint of compact support can be removed.
Moreover, for a class of quasilinear wave equations \eqref{quasi-equation} with the initial data satisfying \eqref{InitialData1} or \eqref{InitialData2}, we can transform them into fully nonlinear cases, hence the compact support constraint is removed and we can still obtain the global well-posedness. In particular, we can even prove the uniform bound of the highest-order energy for the equation \eqref{quasi-equation-standard}.

From now on, we review some related results concerning nonlinear wave equations and elastodynamics.
When $n\geq4$, the global existence of small solutions to the Cauchy problem of quasilinear wave equations is easy and can be obtained by the fact of subcritical decay, for instance, see \cite{Klainerman85, KLainermanPonce83, Shatah}.
The nontrivial long-time behavior was firstly established by John and Klainerman in \cite{JohnKlainerman84} where they showed the almost global existence for 3D quasilinear wave equations. In general, this is sharp \cite{John81}. By introducing the null condition, the global existence results were obtained firstly by Klainerman \cite{Klainerman86} and by Christodoulou \cite{Christodoulou86} independently. Klainerman's proof uses the vector field theory and generalized energy method, while Chritodoulou's argument relies on conformal mapping method. For nonrelativistic systems of nonlinear wave equation where the Lorentz invariance is not available, Klainerman and Sideris introduced the weighted $L^2$ energy with the use of only invariance of translation, rotation, scaling and obtained the almost global existence in 3D \cite{KlainermanSiderisS96} (see an earlier proof by John \cite{John88} using $L^1-L^\infty$ estimate). Subsequently, Sideris adapted the weighted energy method to get the global existence for 3D elastic wave under null condition \cite{Sideris96}. Be of importance is that Sideris \cite{Sideris00} and Agemi \cite{Agemi00} got the global existence under nonresonance null condition which is physically compatible with the system. For 3D incompressible elastodynamics, the global existence was obtained by Sideris and Thomases \cite{SiderisThomas05, SiderisThomas07}. For wave systems with different speed, the global well-posedness were proved by Sideris and Tu in \cite{SiderisTu01}; see also \cite{Yokoyama} for a different method. While there is no null condition, the finite time blow-up was shown for nonlinear wave equations \cite{Alinhac00b,John81} and for compressible elastodynamics \cite{JohnBook,T98}  in three dimension.

In 2D, the blow-up results for quasilinear wave equations were firstly shown by Alinhac in \cite{Alinhac99b,Alinhac99a,Alinhac01b}. Compared with
the three dimensional case, the corresponding long time existence problem is more delicate since quadratic terms have only critical decay even under the assumption of null condition. In the semilinear case, Godin got the global existence under null condition \cite{Godin93}.
In the quasilinear case, if the nonlinearities are cubic terms and they satisfy the null condition, a series of global existence results were derived by \cite{Hoshiga95,Katayama93,Katayama95, ZhouY92} etc. When the nonlinearities contain the quadratic terms, Alinhac \cite{Alinhac01a} was the first to prove the
global existence under both null condition. In Alinhac's proof, the key point is the introduction of ghost weight. However, the method used in \cite{Alinhac01a} heavily relies on the assumption of compact support to the initial data, since a certain kind of Hardy-type inequality depending on the compact support of initial data is crucial \cite{LSZ13}. While there is no Lorentz invariance, the corresponding problem is more complicated. The first non-trivial long time behavior results was established by Lei, Sideris and Zhou \cite{LSZ13} where the authors established the almost global existence for incompressible elastodynamics in Eulerian coordinates. The global well-posedness is finally established by Lei \cite{Lei14} in which the author found a kind of inherent ``strong null condition'' in Lagrangian coordinates
(see a new proof by Wang using space time resonance method \cite{WangXueCheng2014}). We remark the results in \cite{Lei14,LSZ13, WangXueCheng2014} don't require the compact support of the initial data.

The paper is organized as follows. In Section 2, we give some notations and necessary lemmas which are important for our energy estimates. In Section 3, we focus on the null structure of the nonlinearities. Section 4 concerns the higher and lower order energy estimates for \eqref{wave-equation} which yield the first main result of the paper. Then the second main result for the global solutions of \eqref{quasi-equation} is proved in Section 5. In the last section, we show the uniform bound of the highest-order energy for the equation \eqref{quasi-equation-standard}.

\section{Preliminaries}
In the whole paper, we use the following notation conventions. Points in space-time $\mathbb{R}^+ \times \mathbb{R}^2$ are denoted by $$X=(t,x)=(t,x_1,x_2).$$
Partial derivatives are written as
$$
\partial_0=\partial_t=\frac{\partial}{\partial t},\quad \text{and} \quad
\partial_i=\frac{\partial}{\partial x_i},\quad 1\leq i\leq2.
$$
We also abbreviate space derivative and space-time derivative as
$$
\nabla=(\partial_1,\partial_2), \quad \text{and}  \quad
\partial=(\partial_t,\partial_1,\partial_2).
$$
For the convenience, we denote
$$
r=|x|, \quad \omega=\frac xr, \quad \omega^\perp= (\omega^\perp_1,\omega^\perp_2)=(-\omega_2,\omega_1),
\quad \langle a\rangle=\sqrt{1+a^2}.
$$
We often decompose space derivative into radial and angular components
\begin{equation}\label{spatial-decomposition}
\nabla=\omega \partial_r+\frac{\omega^{\perp}}{r}\partial_\theta,
\end{equation}
where $\partial_r=\omega\cdot\nabla, \partial_\theta=x^\perp \cdot \nabla$.
This fact plays an important role in our argument.

A central role in our paper is the application of the generalized vector field operators,
which was introduced by Klainerman \cite{Klainerman85}:
\begin{eqnarray*}
&&\Omega=-x_2 \partial_1+x_1\partial_2, \\
&&S=t\partial_t+r\partial_r, \\
&&L_i=t\partial_i+x_i\partial_t,\quad 1\leq i\leq2.
\end{eqnarray*}
Note that if $u(t,x)$ is a solution to \eqref{wave-equation}, then $\lambda^{-2}u(\lambda x,\lambda t)$ is also a solution with initial data $(\lambda^{-2}\varphi(\lambda x),\lambda^{-1}\psi(\lambda x) )$ for any $\lambda >0$. Hence more precisely, we define the modified scaling generator $\tilde{S}=S-2$. The seven vector fields used in this paper can be denoted by $\Gamma=(\partial,\Omega,L_1,L_2,\tilde{S})$. For $\Gamma^a u$, we mean
$\Gamma^{a_1}...\Gamma^{a_7}$, where $a$ is multi-index $a=(a_{1},\cdots, a_{7})$.
We also use the abbreviation $\Gamma^ku=\{ \Gamma^a u: |a|\leq k\}$.

Based on the above preparation, we define the generalized energy in line with the wave equations by
\begin{equation*}
E_k(u(t))=\frac 12 \sum_{|a|\leq k-1}\int_{\mathbb{R}^2}  |\partial \Gamma^a u|^2 dx.
\end{equation*}
To describe the space of the initial data, we introduce the following notation (See \cite{Sideris00}):
\begin{equation}\nonumber
\Lambda = \{\nabla, r\partial_r, \Omega\},
\end{equation}
and
\[
H^k_\Lambda =\{(f,g):\sum_{| a  |\le k-1}(\|\Lambda^a f\|_{L^2}+\|\nabla\Lambda^a f\|_{L^2}+\|\Lambda^a g\|_{L^2} )<\infty\},
\]
with the norm
\[
\|(f,g)\|_{H^k_\Lambda}=\sum_{| a |\leq k-1} (\|\nabla\Lambda^a f\|_{L^2}+\|\Lambda^a g\|_{L^2}).
\]
Space $\dot{H}_\Gamma(T)$  is defined by
\[
\dot{H}_\Gamma^k(T)=\{u:[0,T)\to \mathbb{R}|\; \partial u\in\cap_{j=0}^{k-1} C^j([0,T);H^{k-j-1}_\Lambda )\}.
\]
Solutions will constructed in the space $\dot{H}_\Gamma(T)$ with the norm
\[
\sup_{0\le t<T} E_k(t)^{1/2}.
\]
It is obviously that $\dot{H}_\Gamma^k(T) \subseteq C^2([0,T)\times\mathbb{R}^2)$ if $k\geq8$.
This confirms the fact that the solutions we constructed are classical one.

Throughout this paper, we use $A\lesssim B$ to denote $A\leq C B$ for some absolute constant $C$,
whose meaning may change from line to line.

Now we state some preliminary weighted estimates.
\begin{lem}\label{lemK-S}
There holds
\begin{equation}\label{K-S}
\langle t+r \rangle^{\frac 12}\langle t-r \rangle^{\frac 12} |u| \lesssim \sum_{|a| \leq 2}\|\Gamma^a u \|_{L^2},
\end{equation}
provided the right-hand side is finite.
\end{lem}
\begin{proof}
If we replace $\widetilde{S}$ by $S$, then \eqref{K-S} is
a classical inequality by S. Klainerman \cite{Klainerman85}.
Note that $\widetilde{S}$ and $S$ differ each other by a lower order term,
thus the lemma is an easy consequence of the classical  Klainerman-Sobolev inequality.
\end{proof}
The following lemma states the relationship between the ordinary derivatives and the vector field.
\begin{lem}\label{lemGamma}
There holds
\begin{equation}\label{Gamma1}
|(t+r)(\partial_t+\partial_r)u| \leq  |\Gamma u|,
\end{equation}
and
\begin{equation}\label{Gamma2}
|(t-r)\partial u|\lesssim |\Gamma u|.
\end{equation}
\end{lem}
\begin{proof}
If we replace $\widetilde{S}$ by $S$, one can find the result in \cite{Klainerman85}. Note the relation between $\widetilde{S}$ and $S$, one immediately has the lemma.
\end{proof}

Combining Lemma 2.1 with \eqref{Gamma2}, we get
strengthened decay rate for $\|\partial^2 u\|_{L^\infty}$ away from the light cone .
\begin{lem}\label{lemK-S-deriv}
Let $u\in E_4$, then there holds
\begin{eqnarray}\nonumber
\langle t+r \rangle^{\frac 12}\langle t-r \rangle^{\frac 32} |\partial^2 u|
\lesssim \sum_{|a| \leq 3}\|\partial\Gamma^a u \|_{L^2},
\end{eqnarray}
\end{lem}

We have the following lemma which asserts that the null structure is preserved upon the commutation between the $\Gamma$ operators and the wave equations \eqref{wave-equation}. This property would be extremely important in our proof.
\begin{lem}\label{LemmGammaEq}
Assume the nonlinearities of the wave equations \eqref{wave-equation} satisfy the null condition \eqref{null-condition}, then there holds
\begin{equation}\label{wave-Gamma}
\Box \Gamma^a u=\sum_{b+c+d = a}  N_d(\Gamma^b u,\Gamma^c u),
\end{equation}
where each term $N_d$ is of the form \eqref{nonlinearities}
 satisfying  \eqref{null-condition}, especially $N_0(u,v)=N(u,v)$.
\end{lem}
\begin{proof}
See L. H\"{o}rmander \cite{HormanderBook}.
\end{proof}

The local existence of the classical solutions to \eqref{wave-equation} and \eqref{quasi-equation} are trivial by standard method,
we omit the details here.
In order
to get the global existence result, it suffices to establish the following \textit{a priori} estimates
\begin{eqnarray}
\frac{d}{dt}E_k(t)\lesssim \langle t\rangle^{-1} E_k(t)E^{\frac 12}_{k-1}(t),
\label{HighOrderEnergyEst}\\
\frac{d}{dt}E_{k-1}(t)\lesssim \langle t\rangle^{-\frac 32} E_k^{\frac 12}(t)E_{k-1}(t)
\label{LowOrderEnergyEst}.
\end{eqnarray}
Once the above estimates are obtained, the main results hold by standard continuity method. For the details, one can consult \cite{Lei14}.

\section{Estimate for the Nonlinearities}
In this section, we are going to study the good properties of the nonlinearities due to the null condition.
Both the higher-order energy estimate and the lower-order energy estimate will benefit from those properties.
We will state them in two different lemmas.

The use of \eqref{null-condition} for the lower-order energy estimate is captured in the following lemma. It says that we are able to gain $\langle t\rangle^{-1}$ decay near the light cone.
\begin{lem} \label{Lemma1}
Let $1\lesssim r$. Suppose that the nonlinearities \eqref{nonlinearities} satisfy \eqref{null-condition}.
Then there holds
\begin{equation}\label{nonlinearity-estimate}
|N(u,v)|
\lesssim
\frac{1}{r} \big(|\partial \Gamma u|+|\partial u|\big)   ( |\partial \Gamma v|+|\partial v| ).
\end{equation}
\end{lem}
\begin{proof}
First, we introduce the notations:
$$D^\pm=\frac{1}{2}(\partial_t\pm\partial_r),\quad Y^\pm=(1,\pm\omega).$$
Consequently, one has
$(\partial_t,\omega \partial_r)=Y^-D^- + Y^+D^+$.
By \eqref{spatial-decomposition}, one can decompose the space-time derivatives
as follows:
\begin{equation} \label{DecomDeriCone}
\partial=(\partial_t,\partial_1,\partial_2)=Y^-D^- + Y^+D^+ +(0,\frac{\omega^{\perp}}{r}\partial_\theta)=Y^-D^- +R,
\end{equation}
where
\begin{equation*}
R=Y^+D^+ +(0,\frac{\omega^{\perp}}{r}\partial_\theta).
\end{equation*}
Here the derivative $D^-$ denotes the bad derivative, $R$ denotes the good derivative.
Simple calculation shows
\begin{equation}
|Ru| \lesssim \text{min}\big\{|\partial u|,|(\partial_t+\partial_r)u|+ |\frac{1}{r}\partial_\theta u| , |\frac{1}{r}\Gamma u| \big\} \label{non-estimate-1},
\end{equation}
and
\begin{equation}
|\partial Ru|,|R^2u|,|D^{\pm} Ru|,|R (Y^{\pm} D^{\pm}) u|
\lesssim \frac{1}{r}(|\partial\Gamma u|+\partial u| ). \label{non-estimate-2}
\end{equation}
Here $\partial Ru,\ R^2u,\ R (Y^{\pm} D^{\pm}) u$ are understood as matrix operators
$\partial\otimes Ru,\ R\otimes Ru,\ R \otimes (Y^{\pm} D^{\pm}) u$, respectively.

Employing the decomposition \eqref{DecomDeriCone}, we organize $N(u,v)$ as follows:
\begin{align}\label{NonLow}
N(u,v)
&=N_{\alpha\beta\mu\nu}
\big\{Y_\alpha^-Y_\beta^- D^-D^- u +  Y_\alpha^-D^-R_\beta u+ R_\alpha(Y^-_\beta D^- u)  +R_\alpha R_\beta u
\big\}                                              \nonumber\\
&\qquad \times\big\{Y_\mu^-Y_\nu^- D^-D^- v +  Y_\mu^-D^-R_\nu v+ R_\mu(Y^-_\nu D^- v)  +R_\mu R_\nu v
\big\}                                                \nonumber\\
&=N_{\alpha\beta\mu\nu} Y_\alpha^-Y_\beta^- Y_\mu^-Y_\nu^- D^-D^- u D^-D^- v  \nonumber\\
&\quad +
N_{\alpha\beta\mu\nu} Y_\alpha^-Y_\beta^- D^-D^- u    [Y_\mu^-D^-R_\nu v+ R_\mu(Y^-_\nu D^- v)  +R_\mu R_\nu v ]  \nonumber\\
&\quad +
N_{\alpha\beta\mu\nu} [Y_\alpha^-D^-R_\beta u+ R_\alpha(Y^-_\beta D^- u) +R_\alpha R_\beta u]  Y_\mu^-Y_\nu^- D^-D^- v\nonumber\\
&\quad  + N_{\alpha\beta\mu\nu} [Y_\alpha^-D^-R_\beta u+ R_\alpha(Y^-_\beta D^- u)  +R_\alpha R_\beta u ]  \nonumber\\
&\qquad\times  [Y_\mu^-D^-R_\nu v+ R_\mu(Y^-_\nu D^- v)  +R_\mu R_\nu v ].
\end{align}
Note $Y^-\in \Sigma$, then thanks to the null condition \eqref{null-condition}, one
immediately has $$N_{\alpha\beta\mu\nu} Y_\alpha^-Y_\beta^- Y_\mu^-Y_\nu^- D^-D^- u D^-D^- v=0.$$
For the last four lines in \eqref{NonLow}, by \eqref{non-estimate-1} and \eqref{non-estimate-2},
it's easy to see that they are bounded by the right hand side of  \eqref{nonlinearity-estimate}.
Thus the lemma is proved.
\end{proof}

We cannot directly use Lemma \ref{Lemma1} in the higher-order energy estimate since it will cause derivative loss problems.
This loss makes the use of the null condition rather delicate in the higher order energy estimate.
Fortunately, we still have the following lemma:
\begin{lem}  \label{Lemma2}
Let $F_1(u)= N_{\alpha\beta\mu\nu}
  \partial_\alpha \Gamma^a  u \partial_\beta \Gamma^a  u \partial_{\mu} \partial_\nu u$,
   $F_2(u)= N_{\alpha\beta\mu\nu}
 \partial_\alpha \Gamma^a  u \partial_\beta \Gamma^a  u \partial_{\mu} \partial_\nu  \partial_t u$,
  $F_3(u)= N_{\alpha\beta\mu\nu}
  \partial_\beta \Gamma^a  u \partial_t \Gamma^a  u \partial_\alpha \partial_{\mu} \partial_\nu u$,
and $1\lesssim r$. Then for all multi-index $a$, there hold
\begin{eqnarray*} \label{Nonlinearity2}
&&|F_1(u)|\lesssim  \sum_{i=1}^2|(\omega_i\partial_t+\partial_i) \Gamma^a  u| |\partial \Gamma^a  u| |\partial^2u|
  +  \frac{1}{r}|\partial_t \Gamma^a  u|^2 ( |\partial \Gamma u|+|\partial u| ), \\
&&|F_2(u)|\lesssim  \sum_{i=1}^2|(\omega_i\partial_t+\partial_i) \Gamma^a  u| |\partial \Gamma^a  u| |\partial^3u|
 +  \frac{1}{r}|\partial_t \Gamma^a  u|^2 ( |\partial^2 \Gamma u|+|\partial^2 u| ) , \\
&&|F_3(u)|\lesssim \sum_{i=1}^2|(\omega_i\partial_t+\partial_i) \Gamma^a  u| |\partial_t \Gamma^a  u| |\partial^3u|
 +  \frac{1}{r}|\partial_t \Gamma^a  u|^2 ( |\partial \Gamma^2 u| +|\partial \Gamma u|+ |\partial u| ).
\end{eqnarray*}
\end{lem}
The nonlinear terms $F_1,\ F_2$ and $F_3$  are taken from the higher-order energy estimate
in the next section.
The lemma says that due to the null condition,
we can estimate $F_1,\ F_2$ and $F_3$  by good derivative $\partial_t+\partial_r$ with some
good remainders .
\begin{proof}
Since the proof of $F_1, F_2, F_3$ is similar, we only give the details concerning $F_1$ and leave $F_2$ and $F_3$ to the interested readers.

\textbf{Case a:} all of $\alpha, \beta, \mu, \nu \in \{1,2\}$. By \eqref{spatial-decomposition}, we rearrange $F$ as follows:
\begin{eqnarray*}
&&N_{\alpha\beta\mu\nu}
  \partial_\alpha \Gamma^a  u \partial_\beta \Gamma^a  u  \partial_{\mu} \partial_{\nu} u   \\
&&= \big\{N_{\alpha\beta\mu\nu}
(\omega_\alpha\partial_t+\partial_\alpha) \Gamma^a  u \partial_\beta \Gamma^a  u  \partial_{\mu} \partial_{\nu} u
- N_{\alpha\beta\mu\nu}
\omega_\alpha \partial_t \Gamma^a  u (\omega_\beta\partial_t+\partial_\beta) \Gamma^a  u  \partial_{\mu} \partial_{\nu} u \big\}\\
&&\quad+N_{\alpha\beta\mu\nu}
\omega_\alpha\omega_\beta\omega_\mu\omega_\nu \partial_t \Gamma^a  u \partial_t \Gamma^a  u  \partial_r \partial_r u \\
&&\quad+ \big\{N_{\alpha\beta\mu\nu} \omega_\alpha\omega_\beta  \partial_t \Gamma^a  u\partial_t \Gamma^a  u
\Large[ \omega_\mu\partial_r(\frac{1}{r}\omega^\perp_\nu\partial_\theta)u+\frac{1}{r}\omega^\perp_\mu\partial_\theta
(\omega_\nu\partial_r+\frac{1}{r}\omega^\perp_\nu\partial_\theta)u \Large] \big\}\\
&&=J_{11}+J_{12}+J_{13}.
\end{eqnarray*}
Simple calculation shows
$$|J_{11}| \lesssim \sum_{i=1}^2|(\omega_i\partial_t+\partial_i) \Gamma^a  u| |\partial \Gamma^a  u| |\partial^2u|,$$
$$|J_{13}| \lesssim
\frac{1}{r}|\partial_t \Gamma^a  u|^2 ( |\partial \Gamma u|+|\partial u| )  .$$

\textbf{Case b:} one of $\alpha, \beta, \mu, \nu $ is 0.

$1)$ $\alpha=0, \beta, \mu, \nu \in \{1,2\}$ or $\beta=0, \alpha, \mu, \nu \in \{1,2\}$.
By the symmetry of the nonlinearities and \eqref{spatial-decomposition}, one gets
\begin{eqnarray*}
&& N_{0\beta\mu\nu} \partial_t \Gamma^a  u \partial_\beta \Gamma^a  u \partial_{\mu} \partial_{\nu} u +
 N_{\alpha 0\mu\nu} \partial_\alpha \Gamma^a  u \partial_t \Gamma^a  u \partial_{\mu} \partial_{\nu} u \\
&&=2 N_{0\beta\mu\nu} \partial_t \Gamma^a  u \partial_\beta \Gamma^a  u \partial_{\mu} \partial_{\nu} u \\
&&=2 N_{0\beta\mu\nu}
 \partial_t \Gamma^a  u (\omega_\beta\partial_t+\partial_\beta) \Gamma^a  u  \partial_{\mu} \partial_{\nu} u\\
&&\quad-2N_{0\beta\mu\nu}
\omega_\beta\omega_\mu\omega_\nu \partial_t \Gamma^a  u \partial_t \Gamma^a  u  \partial_r \partial_r u \\
&&\quad+\big\{- 2N_{0\beta\mu\nu}  \omega_\beta \partial_t \Gamma^a  u\partial_t \Gamma^a  u
\Large[ \omega_\mu\partial_r(\frac{1}{r}\omega^\perp_\nu\partial_\theta)u+\frac{1}{r}\omega^\perp_\mu\partial_\theta
(\omega_\nu\partial_r+\frac{1}{r}\omega^\perp_\nu\partial_\theta)u \Large] \big\}\\
&&=J_{21}+J_{22}+J_{23}.
\end{eqnarray*}
Similarly,
\begin{equation*}
|J_{21}| \lesssim \sum_{i=1}^2|(\omega_i\partial_t+\partial_i) \Gamma^a  u| |\partial \Gamma^a  u| |\partial^2u|,
\end{equation*}
\begin{equation*}
|J_{23}| \lesssim
\frac{1}{r}|\partial_t \Gamma^a  u|^2 ( |\partial \Gamma u|+|\partial u| )  .
\end{equation*}

$2)$ $\mu=0, \alpha, \beta, \nu \in \{1,2\}$ or $\nu=0, \alpha, \beta, \mu \in \{1,2\}$. By the symmetry of the nonlinearities and \eqref{spatial-decomposition}, we have
\begin{eqnarray*}
&& N_{\alpha\beta0\nu}
 \partial_\alpha \Gamma^a  u \partial_\beta \Gamma^a  u  \partial_{t} \partial_{\nu} u
+ N_{\alpha\beta\mu0}
 \partial_\alpha \Gamma^a  u \partial_\beta \Gamma^a  u  \partial_{t} \partial_{\mu} u\\
&&=2 N_{\alpha\beta0\nu}
 \partial_\alpha \Gamma^a  u \partial_\beta \Gamma^a  u  \partial_{t} \partial_{\nu} u\\
&&= \big\{2N_{\alpha\beta0\nu}
 (\omega_\alpha\partial_t+\partial_\alpha) \Gamma^a  u \partial_\beta \Gamma^a  u  \partial_{t} \partial_{\nu} u
- 2N_{\alpha\beta0\nu}
 \omega_\alpha\partial_t \Gamma^a  u (\omega_\beta\partial_t+\partial_\beta) \Gamma^a  u  \partial_{t} \partial_{\nu} u \big\}\\
&&\quad+ 2N_{\alpha\beta0\nu}
 \omega_\alpha\omega_\beta\partial_t \Gamma^a  u \partial_t \Gamma^a  u      \partial_{\nu}(\partial_{t}+\partial_{r}) u\\
&&\quad - 2N_{\alpha\beta0\nu}
  \omega_\alpha\omega_\beta\omega_\nu\partial_t \Gamma^a  u \partial_t \Gamma^a  u  \partial_{r} \partial_r u\\
&&\quad - 2N_{\alpha\beta0\nu}
  \omega_\alpha\omega_\beta\partial_t \Gamma^a  u \partial_t \Gamma^a  u
  (\frac{\omega^\perp_{\nu}}{r}\partial_\theta)\partial_{r} u\\
 &&= J_{31}+J_{32}+J_{33}+J_{34}.
\end{eqnarray*}
As before,
\begin{equation*}
|J_{31}| \lesssim \sum_{i=1}^2|(\omega_i\partial_t+\partial_i) \Gamma^a  u| |\partial \Gamma^a  u| |\partial^2u|.
\end{equation*}
Exploiting the good derivative \eqref{Gamma1}, we get
\begin{equation*}
|J_{32}+J_{34}| \lesssim
\frac{1}{r}|\partial_t \Gamma^a  u|^2 ( |\partial \Gamma u|+|\partial u| ) .
\end{equation*}

\textbf{Case c:} two of $\alpha, \beta, \mu, \nu $ is 0.

$1)$ $\mu, \nu=0, \alpha, \beta \in \{1,2\}$. By \eqref{spatial-decomposition}, then
\begin{eqnarray*}
&& N_{\alpha\beta00}
 \partial_\alpha \Gamma^a  u \partial_\beta \Gamma^a  u \partial_{t} \partial_{t} u\\
&&= \big\{N_{\alpha\beta00}
 (\omega_\alpha\partial_t+\partial_\alpha) \Gamma^a  u \partial_\beta \Gamma^a  u  \partial_{t} \partial_{t} u
- N_{\alpha\beta00}
\omega_\alpha\partial_t \Gamma^a  u (\omega_\beta\partial_t+\partial_\beta) \Gamma^a  u \partial_{t} \partial_{t} u \big\}\\
&&\quad+\big\{ N_{\alpha\beta00}
 \omega_\alpha\omega_\beta\partial_t \Gamma^a  u \partial_t \Gamma^a  u
 \Large[\partial_{t} (\partial_{t} +\partial_{r} )u - \partial_{r} (\partial_{t} +\partial_{r} )u \Large] \big\}\\
&&\quad+ N_{\alpha\beta00}
\omega_\alpha\omega_\beta\partial_t \Gamma^a  u \partial_t \Gamma^a  u \partial_r \partial_r u          \\
&&=J_{41}+J_{42}+J_{43}.
\end{eqnarray*}
We have
\begin{equation*}
|J_{41}| \lesssim \sum_{i=1}^2|(\omega_i\partial_t+\partial_i) \Gamma^a  u| |\partial \Gamma^a  u| |\partial^2u|.
\end{equation*}
By \eqref{Gamma1}, we get
\begin{equation*}
|J_{42}| \lesssim \frac{1}{r}|\partial_t \Gamma^a  u|^2 ( |\partial \Gamma u|+|\partial u| ) .
\end{equation*}
$2)$ $\alpha, \beta=0, \mu, \nu \in \{1,2\}$. By \eqref{spatial-decomposition}, then
\begin{eqnarray*}
&& N_{00\mu\nu}
 \partial_t \Gamma^a  u \partial_t \Gamma^a  u  \partial_{\mu} \partial_{\nu} u\\
&&=N_{00\mu\nu}
\omega_\mu\omega_\nu \partial_t \Gamma^a  u \partial_t \Gamma^a  u  \partial_r \partial_r u \\
&&\quad+ \big\{N_{00\mu\nu}  \partial_t \Gamma^a  u\partial_t \Gamma^a  u
\Large[ \omega_\mu\partial_r(\frac{1}{r}\omega^\perp_\nu\partial_\theta)u+\frac{1}{r}\omega^\perp_\mu\partial_\theta
(\omega_\nu\partial_r+\frac{1}{r}\omega^\perp_\nu\partial_\theta)u \Large] \big\}\\
&&=J_{51}+J_{52}.
\end{eqnarray*}
Obviously,
\begin{equation*}
|J_{52}| \lesssim
\frac{1}{r}|\partial_t \Gamma^a  u|^2 ( |\partial \Gamma u|+|\partial u| ) .
\end{equation*}

$3)$ $\beta, \nu=0, \alpha, \mu \in \{1,2\}$. By \eqref{spatial-decomposition}, then
\begin{eqnarray*}
&&N_{\alpha0\mu0}
  \partial_t \Gamma^a  u \partial_\alpha \Gamma^a  u  \partial_{\mu} \partial_t u   \\
&&=N_{\alpha0\mu0}
  \partial_t \Gamma^a  u (\omega_\alpha\partial_t+\partial_\alpha) \Gamma^a  u  \partial_{\mu} \partial_t u
-N_{\alpha0\mu0}
  \omega_\alpha \partial_t \Gamma^a  u\partial_t \Gamma^a  u  \partial_\mu (\partial_t+\partial_r) u \\
&&\quad +N_{\alpha0\mu0}
  \omega_\alpha \omega_{\mu}\partial_t \Gamma^a  u\partial_t \Gamma^a  u  \partial_r \partial_r u
+N_{\alpha0\mu0}
  \omega_\alpha\partial_t \Gamma^a  u\partial_t \Gamma^a  u   \frac{1}{r}\omega^\perp_{\mu} \partial_\theta\partial_r u\\
&&=J_{61}+J_{62}+J_{63}+J_{64}.
\end{eqnarray*}
Obviously,
\begin{equation*}
|J_{61}| \lesssim \sum_{i=1}^2|(\omega_i\partial_t+\partial_i) \Gamma^a  u| |\partial \Gamma^a  u| |\partial^2u|,
\end{equation*}
and by \eqref{Gamma1}, one gets
\begin{equation*}
|J_{62}+J_{64}| \lesssim \frac{1}{r}|\partial_t \Gamma^a  u|^2 ( |\partial \Gamma u|+|\partial u| ).
\end{equation*}

$4)$ $\beta, \mu=0, \alpha, \nu \in \{1,2\}$,

$5)$ $\alpha, \mu=0, \beta, \nu \in \{1,2\}$,

$6)$ $\alpha, \nu=0, \beta, \mu \in \{1,2\}$,

Since the estimates of $4), 5), 6)$ are similar to $3)$ in \textbf{Case c}, we formulate them together as follows:
\begin{eqnarray*}
&&N_{\alpha00\nu}
   \partial_\alpha \Gamma^a  u \partial_t \Gamma^a  u   \partial_t  \partial_{\nu} u
+N_{0\beta0\nu}
  \partial_t \Gamma^a  u \partial_\beta \Gamma^a  u  \partial_t  \partial_{\nu} u
+N_{0\beta\mu0}
  \partial_t \Gamma^a  u \partial_\beta \Gamma^a  u \partial_{\mu} \partial_t u  \\
&&=\big\{N_{\alpha00\nu} \omega_\alpha \omega_{\nu}
   \partial_t \Gamma^a  u \partial_t \Gamma^a  u   \partial_r  \partial_r u
+N_{0\beta0\nu} \omega_\beta \omega_{\nu}
  \partial_t \Gamma^a  u \partial_t \Gamma^a  u  \partial_r  \partial_r u \\
&&\qquad+N_{0\beta\mu0} \omega_\beta \omega_{\mu}
  \partial_t \Gamma^a  u \partial_t \Gamma^a  u  \partial_r \partial_r u \big\} \\
&&\quad+\big\{
3N_{\alpha00\nu}
  \partial_t \Gamma^a  u (\omega_\alpha\partial_t+\partial_\alpha) \Gamma^a  u  \partial_{\nu} \partial_t u
-3N_{\alpha00\nu}
  \omega_\alpha \partial_t \Gamma^a  u\partial_t \Gamma^a  u  \partial_\nu (\partial_t+\partial_r) u \\
&&\qquad+3N_{\alpha00\nu}
  \omega_\alpha\partial_t \Gamma^a  u\partial_t \Gamma^a  u   \frac{1}{r}\omega^\perp_{\nu} \partial_\theta\partial_r u
  \big\}\\
&&=J_{71}+J_{72}.
\end{eqnarray*}
Similarly,
\begin{equation*}
|J_{72}| \lesssim \sum_{i=1}^2|(\omega_i\partial_t+\partial_i) \Gamma^a  u| |\partial \Gamma^a  u| |\partial^2u|+
\frac{1}{r}|\partial_t \Gamma^a  u|^2 ( |\partial \Gamma u|+|\partial u| ) .
\end{equation*}

\textbf{Case d:} three of $\alpha, \beta, \mu, \nu $ are 0.

$1)$ $\alpha, \beta, \mu=0,  \nu  \in\{1,2\}$ or $\alpha, \beta, \nu=0,  \mu  \in\{1,2\}$. Owning to the symmetry of the nonlinearities and \eqref{spatial-decomposition}, there holds
\begin{eqnarray*}
&&N_{000\nu}
  \partial_t \Gamma^a  u \partial_t \Gamma^a  u  \partial_t\partial_{\nu}  u
+N_{00\mu0}
  \partial_t \Gamma^a  u \partial_t \Gamma^a  u  \partial_t\partial_{\mu}  u   \\
&&=2N_{00\mu0}
  \partial_t \Gamma^a  u \partial_t \Gamma^a  u  \partial_t\partial_{\mu}  u   \\
&&=2N_{00\mu0}
  \partial_t \Gamma^a  u \partial_t \Gamma^a  u \frac{\omega_\mu^\perp}{r} \partial_t\partial_\theta  u
+2N_{00\mu0}  \omega_\mu
  \partial_t \Gamma^a  u \partial_t \Gamma^a  u  \partial_r(\partial_r+\partial_t)  u\\
&&\quad-2N_{00\mu0}  \omega_\mu
  \partial_t \Gamma^a  u \partial_t \Gamma^a  u  \partial_r\partial_r  u\\
&&=J_{81}+J_{82}+J_{83}.
\end{eqnarray*}
Similarly, by \eqref{Gamma1}, there holds
\begin{equation*}
|J_{81}+J_{82}| \lesssim \frac{1}{r}|\partial_t \Gamma^a  u|^2 ( |\partial \Gamma u|+|\partial u| )  .
\end{equation*}

$2)$ $\beta, \mu, \nu=0,  \alpha  \in\{1,2\}$ or $\alpha, \mu, \nu=0,  \beta  \in\{1,2\}$,
\begin{eqnarray*}
&&N_{\alpha000}
  \partial_\alpha \Gamma^a  u \partial_t \Gamma^a  u  \partial_t\partial_t  u
  +N_{0\beta00}
  \partial_t \Gamma^a  u \partial_\beta \Gamma^a  u  \partial_t\partial_t  u   \\
&&=2N_{\alpha000}
  \partial_\alpha \Gamma^a  u \partial_t \Gamma^a  u  \partial_t\partial_t  u   \\
&&=2N_{\alpha000}
  (\omega_\alpha\partial_t+\partial_\alpha) \Gamma^a  u \partial_t \Gamma^a  u  \partial_t\partial_t  u
- 2N_{\alpha000}
  \omega_\alpha\partial_t \Gamma^a  u \partial_t \Gamma^a  u  \partial_t(\partial_t+\partial_r)  u \\
 &&\quad+ 2N_{\alpha000}
  \omega_\alpha\partial_t \Gamma^a  u \partial_t \Gamma^a  u  (\partial_t+\partial_r)\partial_r  u
 - 2N_{\alpha000}
  \omega_\alpha\partial_t \Gamma^a  u \partial_t \Gamma^a  u  \partial_r\partial_r  u  \\
&&=J_{91}+J_{92}+J_{93}+J_{94}.
\end{eqnarray*}
Easily,
\begin{equation*}
|J_{91}| \lesssim \sum_{i=1}^2|(\omega_i\partial_t+\partial_i) \Gamma^a  u| |\partial \Gamma^a  u| |\partial^2u|,
\end{equation*}
and by \eqref{Gamma1}, one gets
\begin{equation*}
|J_{92}+J_{93}| \lesssim \frac{1}{r}|\partial_t \Gamma^a  u|^2 ( |\partial \Gamma u|+|\partial u| )  .
\end{equation*}
\textbf{Case e:} $\alpha=\beta=\mu=\nu=0 $.
\begin{eqnarray*}
&&N_{0000}
\partial_t \Gamma^a  u \partial_t \Gamma^a  u  \partial_t\partial_t  u   \\
&&= N_{0000}
 \partial_t \Gamma^a  u \partial_t \Gamma^a  u  \partial_t(\partial_t+\partial_r)  u
  - N_{0000}
  \partial_t \Gamma^a  u \partial_t \Gamma^a  u  (\partial_t+\partial_r)\partial_r  u\\
&&\quad+ N_{0000}
 \partial_t \Gamma^a  u \partial_t \Gamma^a  u  \partial_r\partial_r  u  \\
&&=J_{01}+J_{02}+J_{03}.
\end{eqnarray*}
Obviously, by \eqref{Gamma1}, there holds
\begin{equation*}
|J_{01}+J_{02}| \lesssim \frac{1}{r}|\partial_t \Gamma^a  u|^2 ( |\partial \Gamma u|+|\partial u| ) .
\end{equation*}
Now we have taken care of all the terms except for
$J_{12},\ J_{22},\ J_{33},\ J_{43},\ J_{51},\ J_{63},\ J_{71},\ J_{83},\ J_{94}$ and $J_{03}$.
Each of them can not be bounded by some good derivative.
However, note that the sum of them is
\begin{eqnarray*}
&&J_{12}+J_{22}+J_{33}+J_{43}+J_{51}+J_{63}+J_{71}+J_{83}+J_{94}+J_{03}\\
&&=N_{\alpha\beta\mu\nu}X_\alpha X_\beta X_\mu X_\nu \partial_t \Gamma^a  u \partial_t \Gamma^a  u  \partial_r\partial_r u.
\end{eqnarray*}
Here $X=(-1,\omega_1,\omega_2)\in\Sigma$,
thus they vanish by the null condition \eqref{null-condition}. This completes the proof of the lemma.
\end{proof}

\section{Energy Estimate}
This section is devoted to the energy estimate. We split the proof into two subsections, which correspond to the higher-order energy estimate and the lower-order energy estimate, respectively.

In Theorem \ref{thm}, by taking appropriate small $\epsilon_0$, we can assume
$E_{k-1}\ll 1$, which will be always assumed in the following argument.

\subsection{Higher-order Energy Estimate}
In this subsection, we perform the higher-order energy estimate.
Apart from the usual energy estimate for wave equations, we will
see that the ghost weight method of Alinhac plays an important role in our argument.

Let $k\geq8$, $|a |\leq k-1$, $\sigma =r-t$ and $q(\sigma)=\arctan \sigma$. We write $e^q=e^{q(\sigma)}$ for simplicity of presentation. Taking the $L^2$ inner product of the equations \eqref{wave-Gamma} with
$e^{q}\partial_t \Gamma^a  u $. Then employing integration by parts, we have
\begin{eqnarray}\label{oo}
&&\frac12\frac{d}{dt} \int_{\mathbb{R}^2}
{e^{q}(|\partial_t\Gamma^a  u|^2+|\nabla\Gamma^a  u|^2 )} dx \nonumber\\
&&+\frac12\sum_{i=1}^2\int_{\mathbb{R}^2}{ \frac{e^{q} }{1+\sigma^2}      |(\omega_i\partial_t+\partial_i)\Gamma^a  u|^2 } dx\nonumber\\
&&=\int_{\mathbb{R}^2} \sum_{b +c+d = a }
N_d(\Gamma^b  u,\Gamma^c u) \partial_t \Gamma^a  u e^qdx.
\end{eqnarray}
At first sight, we will always lose one derivative for the highest order terms.
Fortunately, the symmetry of the nonlinearities enables us to circumvent this problem.
The price we pay is to lose $\langle t \rangle^{-1}$ decay rate, which will be
compensated by the ghost weight method.
We remark that the estimate for the highest order terms is the only spot
where the ghost weight energy is used.

For the highest order term $N_0(\Gamma^b  u,\Gamma^c u)$ in \eqref{oo}, we only treat the case for $b=a$.
The counterpart case for $c=a$ can be estimated exactly in the same way, we omit the details here.
 Using integration by parts, one has
\begin{align}\label{high-oder-1}
&\int_{\mathbb{R}^2}  N_0(\Gamma^a  u, u) \partial_t \Gamma^a  u e^{q}dx\nonumber\\
&=\int_{\mathbb{R}^2}
N_{\alpha\beta\mu\nu} \partial_\alpha  \partial_\beta  \Gamma^a  u  \partial_{\mu} \partial_{\nu} u
\partial_t \Gamma^a  u e^q dx \nonumber\\
&=\int_{\mathbb{R}^2}
N_{\alpha\beta\mu\nu} \partial_\alpha  ( \partial_\beta  \Gamma^a  u  \partial_{\mu} \partial_{\nu} u
\partial_t \Gamma^a  u e^q )dx
-\int_{\mathbb{R}^2}
N_{\alpha\beta\mu\nu}  \partial_\beta  \Gamma^a  u
 \partial_\alpha  \partial_{\mu} \partial_{\nu} u\partial_t \Gamma^a  u e^q dx \nonumber\\
&\quad-\int_{\mathbb{R}^2}
N_{\alpha\beta\mu\nu}  \partial_\beta  \Gamma^a  u
\partial_{\mu} \partial_{\nu} u\partial_\alpha \partial_t \Gamma^a  u e^q dx
-\int_{\mathbb{R}^2}
N_{\alpha\beta\mu\nu}  \partial_\beta  \Gamma^a  u
\partial_{\mu} \partial_{\nu} u\partial_t \Gamma^a  u \partial_\alpha  e^qdx.
\end{align}
The first term on the right hand side of \eqref{high-oder-1} will be
absorbed into the generalized energy.
The second term has no derivative loss problem and the null condition is present.
The third and the forth term need further attention.

For the third term on the right hand side of \eqref{high-oder-1},
by the symmetry of the nonlinearities, we write
\begin{eqnarray}\label{high-oder-2}
&&-\int_{\mathbb{R}^2}
N_{\alpha\beta\mu\nu}  \partial_\beta  \Gamma^a  u
\partial_{\mu} \partial_{\nu} u\partial_\alpha \partial_t \Gamma^a  u e^{q}dx \nonumber \\
&&=-\frac{1}{2}\int_{\mathbb{R}^2}
N_{\alpha\beta\mu\nu}
\partial_{\mu} \partial_{\nu} u \partial_t (\partial_\beta  \Gamma^a  u \partial_\alpha  \Gamma^a  u ) e^q dx\nonumber\\
&&=-\frac{1}{2} \partial_t\int_{\mathbb{R}^2} N_{\alpha\beta\mu\nu}
\partial_{\mu} \partial_{\nu} u \partial_\beta  \Gamma^a  u \partial_\alpha  \Gamma^a  u e^q dx
+\frac{1}{2} \int_{\mathbb{R}^2} N_{\alpha\beta\mu\nu}
\partial_t \partial_{\mu} \partial_{\nu} u \partial_\beta  \Gamma^a  u \partial_\alpha  \Gamma^a  u e^qdx \nonumber\\
&&\quad+\frac{1}{2}\int_{\mathbb{R}^2} N_{\alpha\beta\mu\nu}
 \partial_{\mu} \partial_{\nu} u \partial_\beta  \Gamma^a  u \partial_\alpha  \Gamma^a  u \partial_t e^qdx.
\end{eqnarray}
For the fourth term on the right hand side of \eqref{high-oder-1},
we organize them as follows:
\begin{align}\label{high-oder-3}
&
-N_{\alpha\beta\mu\nu}  \partial_\beta  \Gamma^a  u
\partial_{\mu} \partial_{\nu} u\partial_t \Gamma^a  u \partial_\alpha  e^q  \nonumber\\
&=
-N_{0\beta\mu\nu}  \partial_\beta  \Gamma^a  u
\partial_{\mu} \partial_{\nu} u\partial_t \Gamma^a  u \partial_t e^q
-\sum_{i=1}^2N_{i\beta\mu\nu}  \partial_\beta  \Gamma^a  u
\partial_{\mu} \partial_{\nu} u\partial_t \Gamma^a  u \partial_i e^q  \nonumber\\
&\quad
-\sum_{i=1}^2N_{i\beta\mu\nu}  \partial_\beta  \Gamma^a  u
\partial_{\mu} \partial_{\nu} u\partial_i \Gamma^a  u \partial_t e^q
+\sum_{i=1}^2N_{i\beta\mu\nu}  \partial_\beta  \Gamma^a  u
\partial_{\mu} \partial_{\nu} u\partial_i \Gamma^a  u \partial_t e^q  \nonumber\\
&=-N_{\alpha\beta\mu\nu}  \partial_\beta  \Gamma^a  u
\partial_{\mu} \partial_{\nu} u\partial_\alpha  \Gamma^a  u \partial_t e^{q}
-\sum_{i=1}^2N_{i\beta\mu\nu}  \partial_\beta  \Gamma^a  u
\partial_{\mu} \partial_{\nu} u (\omega_i\partial_t+\partial_i) \Gamma^a  u
 \frac{e^{q} }{1+\sigma^2}.
\end{align}
Combining the above \eqref{high-oder-1}$-$\eqref{high-oder-3}, we derive that
\begin{align}
&\int_{\mathbb{R}^2}
N_{\alpha\beta\mu\nu} \partial_\alpha  \partial_\beta  \Gamma^a  u  \partial_{\mu} \partial_{\nu} u
\partial_t \Gamma^a  u e^{q} dx  \nonumber\\
&=\partial_t\int_{\mathbb{R}^2}
N_{0\beta\mu\nu}   \partial_\beta  \Gamma^a  u  \partial_{\mu} \partial_{\nu} u
\partial_t \Gamma^a  u e^{q} dx
-\int_{\mathbb{R}^2}
N_{\alpha\beta\mu\nu} \partial_\beta  \Gamma^a  u
 \partial_\alpha  \partial_{\mu} \partial_{\nu} u\partial_t \Gamma^a  u e^{q} dx \nonumber\\
&\quad -\frac{1}{2} \partial_t\int_{\mathbb{R}^2} N_{\alpha\beta\mu\nu}
\partial_{\mu} \partial_{\nu} u \partial_\beta  \Gamma^a  u \partial_\alpha  \Gamma^a  u e^{q} dx
+\frac{1}{2} \int_{\mathbb{R}^2} N_{\alpha\beta\mu\nu}
\partial_t \partial_{\mu} \partial_{\nu} u \partial_\beta  \Gamma^a  u \partial_\alpha  \Gamma^a  u e^{q} dx \nonumber\\
&\quad-\frac{1}{2}\int_{\mathbb{R}^2} N_{\alpha\beta\mu\nu}
 \partial_{\mu} \partial_{\nu} u \partial_\beta  \Gamma^a  u \partial_\alpha  \Gamma^a  u \partial_te^{q} dx \nonumber\\
&\quad -\sum_{i=1}^2\int_{\mathbb{R}^2} N_{i\beta\mu\nu}  \partial_\beta  \Gamma^a  u
\partial_{\mu} \partial_{\nu} u (\omega_i\partial_t+\partial_i) \Gamma^a  u   \frac{e^{q} }{1+\sigma^2} dx \nonumber\\
&=I_1+I_2+I_3+I_4+I_5+I_6.
\end{align}
In the sequel, we will estimate $I_1$ to $I_6$ one by one.

Denote the ghost weight energy by
\begin{equation*}
G(t)=\sum_{i=1}^2\int_{\mathbb{R}^2}{ \frac{e^{q} }{1+\sigma^2}      |(\omega_i\partial_t+\partial_i)\Gamma^a  u|^2 } dx.
\end{equation*}
In view of Lemma \ref{Lemma2}, we have
\begin{align}\label{Estimate-I2456}
&I_2+I_4+I_5+I_6 \nonumber\\
&\lesssim \int_{r\geq \langle t\rangle /2}
\sum_{i=1}^2\Large|(\omega_i\partial_t+\partial_i) \Gamma^a  u \partial \Gamma^a  u \Large| \Large(|\partial^3u|+|\partial^2u|\Large)\nonumber\\
&\qquad +\frac{1}{r}|\partial_t \Gamma^a  u|^2
\Large( |\partial \Gamma^2 u| +|\partial \Gamma u|+ |\partial u| \Large)dx\nonumber\\
&\quad +\int_{r\leq \langle t\rangle /2}  |\partial \Gamma^a  u|^2 (|\partial^3u|+|\partial^2u| )dx.
\end{align}
Here we have divided the integral domain $\mathbb{R}^2$ into two domains.
By Lemma \ref{lemK-S} and H\"{o}lder inequality, the first two lines
on the right hand side of \eqref{Estimate-I2456} can be estimated by
\begin{eqnarray*}
&&\int_{r\geq \langle t\rangle /2}
\sum_{i=1}^2\Large|(\omega_i\partial_t+\partial_i) \Gamma^a  u \partial \Gamma^a  u \Large| \Large(|\partial^3u|+|\partial^2u|\Large)\nonumber\\
&&\quad +  \frac{1}{r}|\partial_t \Gamma^a  u|^2
\Large( |\partial \Gamma^2 u| +|\partial \Gamma u|+ |\partial u| \Large)dx\nonumber\\
&& \lesssim
\sum_{i=1}^2 \big\| \frac{(\omega_i\partial_t+\partial_i) \Gamma^a  u}{\langle r-t\rangle} \big\|_{L^2}
 \|\partial \Gamma^a  u\|_{L^2} \|\langle r-t\rangle
\large(|\partial^3u|+|\partial^2u|\large)\|_{L^\infty} \nonumber\\
&&\quad +  \langle t\rangle^{-1} \big\| |\partial \Gamma^a  u|^2 \big\|_{L^1} \|\partial \Gamma^2 u\|_{L^\infty} \\
&&\lesssim  C_\eta\langle t \rangle^{-1}E_{k}E_{k-1} +\eta G.
\end{eqnarray*}
Here $\eta$ is a constant which can be chosen to be any positive number,
$C_\eta$ is a constant depends only on $\eta$ and $k$.

For the last line of \eqref{Estimate-I2456},
we derive from Lemma \ref{lemK-S-deriv} that
\begin{eqnarray*}
&&\int_{r\leq \langle t\rangle /2}  |\partial \Gamma^a  u|^2 (|\partial^3u|+|\partial^2u| )dx \nonumber\\
&& \leq
\big\| |\partial \Gamma^a  u|^2 \big\|_{L^1}
\big\| |\partial^3u|+|\partial^2u| \big\|_{L^\infty(r\leq \langle t\rangle /2)} \nonumber\\
&&\lesssim  \langle t\rangle^{-\frac 32}E_{k}E_{k-1}^{\frac 12}.
\end{eqnarray*}
Then inserting the above two estimates into \eqref{Estimate-I2456},
one gets
\begin{equation}\label{I2456}
I_2+I_4+I_5+I_6\lesssim  C_\eta\langle t \rangle^{-1}E_{k}E_{k-1}^{\frac12} +\eta G.
\end{equation}
Here we have used the assumption $E_{k-1}\ll 1$.

Next, we are going to take care of $I_1$ and $I_3$.
They will be absorbed by the energy as a lower order perturbation.
Denote
\begin{eqnarray}\label{Estimate-purturbation1}
&&\tilde{E}_a(u(t))=\frac 12 \int_{\mathbb{R}^2} e^{q(\sigma)} |\partial\Gamma^a  u|^2 dx-
\int_{\mathbb{R}^2}
N_{0\beta\mu\nu}   \partial_\beta  \Gamma^a  u  \partial_{\mu} \partial_{\nu} u
\partial_t \Gamma^a  u e^{q} dx  \nonumber\\
&&\qquad\qquad\quad +\frac{1}{2} \int_{\mathbb{R}^2} N_{\alpha\beta\mu\nu}
\partial_{\mu} \partial_{\nu} u \partial_\beta  \Gamma^a  u \partial_\alpha  \Gamma^a  u e^{q} dx .
\end{eqnarray}
The reason why we define $\tilde{E}_a(u(t))$ is that one moves $I_1$ and $I_3$ to the left hand side of \eqref{oo}, then one will get $\tilde{E}_a(u(t))$. As is indicated in the beginning of this section, we assume $E_{k-1}\ll 1$. Then
\begin{eqnarray}
&&\Large|
\frac{1}{2} \int_{\mathbb{R}^2} N_{\alpha\beta\mu\nu}
\partial_{\mu} \partial_{\nu} u \partial_\beta  \Gamma^a  u \partial_\alpha  \Gamma^a  u e^{q} dx
-\int_{\mathbb{R}^2}
N_{0\beta\mu\nu}   \partial_\beta  \Gamma^a  u  \partial_{\mu} \partial_{\nu} u
\partial_t \Gamma^a  u e^{q} dx
\Large| \nonumber\\
&&\lesssim \int_{\mathbb{R}^2}|\partial\Gamma^a u|^2|\partial^2u| dx
\leq E_{k-1}^{\frac12} \int_{\mathbb{R}^2} |\partial\Gamma^a u|^2 dx
\ll  \int_{\mathbb{R}^2} |\partial\Gamma^a u|^2 dx .
\end{eqnarray}
On the other hand, note that
\begin{equation}
\int_{\mathbb{R}^2}{e^{q}(|\partial_t\Gamma^a  u|^2+|\nabla\Gamma^a  u|^2 )} dx \sim
\int_{\mathbb{R}^2}|\partial_t\Gamma^a  u|^2+|\nabla\Gamma^a  u|^2 dx.
\end{equation}
Hence one deduces that
\begin{equation} \label{Estimate-purturbation2}
\int_{\mathbb{R}^2} |\partial\Gamma^a  u|^2 dx  \sim \tilde{E}_a(u(t)). 
\end{equation}
Combining \eqref{oo}-\eqref{Estimate-purturbation2}, we have
\begin{eqnarray*}
&&\frac{d}{dt} \int_{\mathbb{R}^2}{(|\partial_t\Gamma^a  u|^2+|\nabla\Gamma^a  u|^2 )} dx
+G(t)\\
&&\lesssim \int_{\mathbb{R}^2} \sum_{b +c+d = a ,d\neq0}  N_d(\Gamma^b  u,\Gamma^c u) \partial_t \Gamma^a  u e^{q}dx\\
&&\quad +\langle t \rangle^{-1}E_{k}E_{k-1}^{\frac 12} + \eta G.
\end{eqnarray*}
Let $\eta$ be small enough, then $\eta G$ on the right hand of the above inequality will be absorbed by the left-hand side,
thus we are arriving at
\begin{equation*}
\frac{d}{dt} \int_{\mathbb{R}^2} |\partial \Gamma^a  u|^2 dx
\lesssim \int_{\mathbb{R}^2} \sum_{b +c+d = a ,d\neq0}  N_d(\Gamma^b  u,\Gamma^c u) \partial_t \Gamma^a  u e^{q}dx
+\langle t \rangle^{-1}E_{k}E_{k-1}^{\frac 12}.
\end{equation*}
Now we are going to estimate the remaining lower order terms. We still split the integral domains into two parts.
First, for $r\leq \langle t\rangle /2$, there holds $\langle r-t\rangle \lesssim \langle t\rangle$.
Since $|b| +|c| < |a| $, without loss of generality, we assume $|b | \leq [|a |/2]$. By Lemma \ref{lemK-S-deriv},
we deduce that
\begin{eqnarray*}
&&\sum_{\tiny\begin{matrix} |b|+|c|< |a|\\ d \neq 0 \end{matrix}} \int_{r\leq \langle t\rangle /2}   N_d(\Gamma^b  u,\Gamma^c u) \partial_t \Gamma^a  u e^{q} dx\\
&&\lesssim     \sum_{\tiny\begin{matrix} |b|+|c|< |a|\\|b | \leq [|a |/2] \end{matrix}}
\int_{r\leq \langle t\rangle /2} |\partial_t \Gamma^a  u|  |\partial^2 \Gamma^b  u| |\partial^2 \Gamma^c u| dx\\
&&\lesssim  \sum_{\tiny\begin{matrix} |b|+|c|< |a|\\|b | \leq [|a |/2] \end{matrix}}
            \|\partial_t \Gamma^a  u\|_{L^2(r\leq \langle t\rangle /2)}
            \|\partial^2 \Gamma^b  u\|_{L^\infty(r\leq \langle t\rangle /2)}
            \|\partial^2 \Gamma^c u\|_{L^2(r\leq \langle t\rangle /2)}    \\
&&\lesssim  \langle t \rangle^{-2}E_k E_{k-1}^{\frac 12}.
\end{eqnarray*}
When $r\geq \langle t\rangle /2$, we need to fully utilize the null condition of the nonlinearities. By Lemma \ref{nonlinearity-estimate} and Lemma \ref{lemK-S}, one has
\begin{eqnarray*}
&&\sum_{\tiny\begin{matrix} |b|+|c|< |a|\\d\neq 0 \end{matrix}}\int_{r\geq \langle t\rangle /2}   N_d(\Gamma^b  u,\Gamma^c u) \partial_t \Gamma^a  u e^{q} dx\\
&&
\lesssim\sum_{|b| +|c| < |a| }\int _{r\geq \langle t\rangle /2}
\frac{1}{r}  |\partial \Gamma^a u|
|\partial \Gamma^{|b| +1} u| |\partial \Gamma^{|c|+1} u| dx\\
&&\lesssim
\sum_{\tiny\begin{matrix} |b|+|c|< |a|\\|b | \leq [|a |/2] \end{matrix}}
\langle t \rangle^{-1}
\|\partial \Gamma^a u\|_{L^2}   \|\partial \Gamma^{|b| +1} u\|_{L^\infty}  \|\partial \Gamma^{|c|+1} u\|_{L^2} \\
&&\lesssim  \langle t \rangle^{-\frac 32}E_k E_{k-1}^{\frac 12}.
\end{eqnarray*}
It then follows from the above three estimates that
\begin{equation*}
\frac{d}{dt} \int_{\mathbb{R}^2} |\partial \Gamma^a  u|^2 dx
\lesssim \langle t \rangle^{-1}E_{k}E_{k-1}^{\frac 12}.
\end{equation*}
Summing over $|a|\leq k-1$ yields
\begin{eqnarray*}
\frac{d}{dt}E_k(t)\lesssim \langle t\rangle^{-1} E_k(t)E^{\frac 12}_{k-1}(t).
\end{eqnarray*}
This gives the first differential inequality \eqref{HighOrderEnergyEst}.

\subsection{Lower-order Energy Estimate}
In this subsection, we perform the lower-order energy estimate.
Unlike the higher-order energy estimate, we care less about
the derivative loss problems since it's not an issue.
We will focus our mind on obtaining the maximal decay in time.

Let $|a |\leq k-2$. Multiplying the equations \eqref{wave-Gamma} with
$\partial_t \Gamma^a  u $, then integrating over $\mathbb{R}^2$, we have
\begin{eqnarray*}
&&\frac{d}{dt} \int_{\mathbb{R}^2}
{\frac12 (|\partial_t\Gamma^a  u|^2+|\nabla\Gamma^a  u|^2 )} dx\\
&&= \int_{\mathbb{R}^2}
 \sum_{b +c+d= a }  N_d(\Gamma^b  u,\Gamma^c u) \partial_t \Gamma^a  u  dx.
\end{eqnarray*}
We still split the integral domains into two parts.  When $r\leq \langle t\rangle /2$, by lemma \ref{lemK-S-deriv}, one gets
\begin{eqnarray*}
&& \int_{r\leq \langle t\rangle /2}
 \sum_{b +c+d = a }  N_d(\Gamma^b  u,\Gamma^c u) \partial_t \Gamma^a  u dx \\
 &&\lesssim     \sum_{\tiny\begin{matrix} |b|+|c|\leq |a|\\|b | \leq [|a |/2] \end{matrix}}
\int_{r\leq \langle t\rangle /2} |\partial_t \Gamma^a  u|  |\partial^2 \Gamma^b  u| |\partial^2 \Gamma^c u| dx\\
&&\lesssim  \sum_{\tiny\begin{matrix} |b|+|c|\leq |a|\\|b | \leq [|a |/2] \end{matrix}}
            \|\partial_t \Gamma^a  u\|_{L^2(r\leq \langle t\rangle /2)}
            \|\partial^2 \Gamma^b  u\|_{L^\infty(r\leq \langle t\rangle /2)}
            \|\partial^2 \Gamma^c u\|_{L^2(r\leq \langle t\rangle /2)}    \\
&&\lesssim  \langle t \rangle^{-2}E_k^{\frac 12} E_{k-1}.
\end{eqnarray*}
When $r\geq \langle t\rangle /2$, by Lemma \ref{Lemma1} and Lemma \ref{lemK-S}, one gets
\begin{eqnarray*}
&&\int_{r\geq \langle t\rangle /2}
 \sum_{b +c+d = a }  N_d(\Gamma^b  u,\Gamma^c u) \partial_t \Gamma^a  u dx \\
&&\lesssim
\sum_{|b|+|c|\leq |a| }\int _{r\geq \langle t\rangle /2}
\frac{1}{r}  |\partial \Gamma^a u|
|\partial \Gamma^{|b| +1} u| |\partial \Gamma^{|c|+1} u| dx\\
&&\lesssim
\sum_{\tiny\begin{matrix} |b|+|c|\leq |a|\\|b | \leq [|a |/2] \end{matrix}}
\langle t \rangle^{-1}
\|\partial \Gamma^a u\|_{L^2}   \|\partial \Gamma^{|b| +1} u\|_{L^\infty}  \|\partial \Gamma^{|c|+1} u\|_{L^2} \\
&&\lesssim  \langle t \rangle^{-\frac 32}E_{k-1}E_k^{\frac 12}.
\end{eqnarray*}
Gathering all the above estimates in this subsection, and summing over $|a|\leq k-1$ yields
\begin{eqnarray*}
\frac{d}{dt}E_{k-1}(t)\lesssim \langle t\rangle^{-\frac 32} E_k^{\frac 12}(t)E_{k-1}(t).
\end{eqnarray*}
This gives the lower-order energy estimate \eqref{LowOrderEnergyEst}.

\section{Proof of Theorem \ref{thm-quasi}}

In this section, we are going to prove Theorem \ref{thm-quasi}. The main idea is to transform \eqref{quasi-equation}
into fully nonlinear ones. For the latter, we have obtained the global well-posedness result in the last section.

First, we give the definition of $\varphi$ and $\psi$ which appear in Theorem \ref{thm-quasi}.
For \eqref{quasi-equation}, if $A_0=0$, we assume $A_1\neq 0$ without loss of generality. Under this case, define
\begin{equation} \label{InitialData1}
\begin{cases}
\varphi=\frac{1}{A_1}\int_{-\infty}^{x_1} v_0\big(s,x_2+\frac{A_2}{A_1}(s-x_1) \big)ds,\\
\psi=\frac{1}{A_1}\int_{-\infty}^{x_1} v_1\big(s,x_2+\frac{A_2}{A_1}(s-x_1) \big)ds.
\end{cases}
\end{equation}
Otherwise, if $A_0\neq 0$, define
\begin{equation}\label{InitialData2}
\begin{cases}
\varphi=\chi,\\
\psi=\frac{1}{A_0}v_0-\frac{A_1}{A_0}\partial_1\chi-\frac{A_2}{A_0}\partial_2\chi,
\end{cases}
\end{equation}
 where $\chi$ is a function satisfing the following equation
\begin{eqnarray}  \label{InitialDataCondition3}
\begin{cases}
(A_1\partial_1 +A_2\partial_2)^2 \chi-A^2_0\Delta \chi \\
\quad=-A_0\partial_t v(0)+A_1\partial_1 v(0)+A_2\partial_2 v(0)
+A^2_0N_{\mu\delta}\partial_\mu v(0) \partial_\delta v(0),\\
(\chi(x),\nabla\chi(x)) \in H^k_\Lambda.
\end{cases}
\end{eqnarray}

Next, we elucidate the nonlinearities of \eqref{quasi-equation}.
Obviously \eqref{quasi-null} implies \eqref{quasi-null2}. On the other side,
if
\begin{equation*}
A_l N_{\mu\delta} X_l X_\mu X_\delta=0
\end{equation*}
for all $X\in \Sigma$, one deduce that
\begin{equation*}
(A_l X_l) (N_{\mu\delta} X_\mu X_\delta)=0.
\end{equation*}
 From  which one must have
\begin{equation*}
N_{\mu\delta} X_\mu X_\delta=0
\end{equation*}
for all $X\in \Sigma$. This means \eqref{quasi-null} and \eqref{quasi-null2} are equivalent.

Following the same argument, one can see that \eqref{quasi-null} is also equivalent to
\begin{equation} \label{quasi-null3}
A_lA_m N_{\mu\delta}X_lX_m  X_\mu X_\delta=0.
\end{equation}
We will directly use the  null condition \eqref{quasi-null3} in the following proof.

Before the proof of Theorem \ref{thm-quasi},
we first state a simple lemma which asserts that the null condition is preserved under symmetrization procedure.
\begin{lemma} \label{lemmaSymmetry}
Suppose that the nonlinearities \eqref{nonlinearities} satisfy \eqref{null-condition}, but they may not satisfy
the symmetry condition \eqref{symmetry-condition}. After some symmetrization procedure, they still satisfy
\eqref{null-condition}.
\end{lemma}
\begin{proof}
Define
\begin{equation*}
\widetilde{N}_{\alpha\beta\mu\nu}=\frac 14 ( N_{\alpha\beta\mu\nu}+ N_{\beta\alpha\mu\nu}+N_{\alpha\beta\nu\mu}+ N_{\beta\alpha\nu\mu}).
\end{equation*}
One can check that
\begin{equation*}
N_{\alpha\beta\mu\nu} \partial_\alpha \partial_\beta u \partial_{\mu} \partial_{\nu}u
=\widetilde{N}_{\alpha\beta\mu\nu} \partial_\alpha \partial_\beta u \partial_{\mu} \partial_{\nu}u.
\end{equation*}
Moreover, $\widetilde{N}_{\alpha\beta\mu\nu}$ satisfy the symmetry \eqref{symmetry-condition} and the null condition
\eqref{null-condition}.
\end{proof}
Now we prove Theorem \ref{thm-quasi}.
\begin{proof}
We are going to show that
\eqref{quasi-equation} can be transformed into
\begin{equation}\label{eq-fully-transform}
\begin{cases}
\Box u=A_\lambda A_\nu N_{\mu\delta}\partial^2_{\lambda\mu} u \partial^2_{\nu\delta} u, \\
u(0,x)=\varphi,\partial_t u(0,x) =\psi,
\end{cases}
\end{equation}
where $(\varphi,\psi)$ is defined by  \eqref{InitialData1} or by \eqref{InitialData2}-\eqref{InitialDataCondition3}.

\textbf{Case a:} $A_0=0$.

 Without loss of generality, we assume $A_1\neq0$, let
\begin{eqnarray}
&&u(t,x)=\frac{1}{A_1}\int_{-\infty}^{x_1} v\big(t,s,x_2+\frac{A_2}{A_1}(s-x_1) \big)ds. \nonumber
\end{eqnarray}
Simple calculation shows
\begin{eqnarray}
&&\partial_1 u(t,x)=\frac{1}{A_1}v(t,x)-\frac{A_2}{A^2_1}\int_{-\infty}^{x_1} \partial_2 v\big(t,s,x_2+\frac{A_2}{A_1}(s-x_1) \big) d s, \nonumber\\
&&\partial_2 u(t,x)=\frac{1}{A_1} \int_{-\infty}^{x_1} \partial_2 v\big(t,s,x_2+\frac{A_2}{A_1}(s-x_1) \big) d s . \nonumber
\end{eqnarray}
One can infer that
\begin{equation}
v=A_i\partial_i u . \nonumber
\end{equation}
Inserting them into \eqref{quasi-equation}, we get
\begin{equation*}
A_i\partial_i( \Box u-A_\lambda A_\nu N_{\mu\delta}\partial^2_{\lambda\mu} u \partial^2_{\nu\delta} u)=0.
\end{equation*}
Let
\begin{equation}\nonumber
 \Box u-A_\lambda A_\nu N_{\mu\delta}\partial^2_{\lambda\mu} u \partial^2_{\nu\delta} u=0.
\end{equation}
Obviously one has
\begin{eqnarray*}
&&u(0,x)=\frac{1}{A_1}\int_{-\infty}^{x_1} v_0\big(s,x_2+\frac{A_2}{A_1}(s-x_1) \big)ds,\\
&&\partial_t u(0,x)=\frac{1}{A_1}\int_{-\infty}^{x_1} v_1\big(s,x_2+\frac{A_2}{A_1}(s-x_1) \big)ds,
\end{eqnarray*}
which gives \eqref{InitialData1}.

\textbf{Case b:} $A_0\neq0$.

Let
\begin{eqnarray}
&&u(t,x)=\frac{1}{A_0}\int_0^t v\big(s,x_1+\frac{A_1}{A_0}(s-t),x_2+\frac{A_2}{A_0}(s-t) \big)ds  \nonumber\\
&&\qquad\qquad +\chi\big(x_1-\frac{A_1}{A_0}t,x_2-\frac{A_2}{A_0}t \big), \nonumber
\end{eqnarray}
where $\chi$ will be chosen later. Simple calculation shows
\begin{eqnarray}
&&\partial_t u(t,x)=\frac{1}{A_0}v(t,x)\nonumber\\
&&\qquad\qquad\quad +\frac{1}{A_0}\int_0^t (-\frac{A_1}{A_0}\partial_1-\frac{A_2}{A_0}\partial_2) v\big(s,x_1+\frac{A_1}{A_0}(s-t),x_2+\frac{A_2}{A_0}(s-t) \big)ds  \nonumber\\
&&\qquad\qquad\quad -(\frac{A_1}{A_0}\partial_1+\frac{A_2}{A_0}\partial_2)\chi\big(x_1-\frac{A_1}{A_0}t,x_2-\frac{A_2}{A_0}t \big), \nonumber\\
&&\partial^2_t u(t,x)=\frac{1}{A_0}\partial_t v(t,x)
-\frac{1}{A_0}(\frac{A_1}{A_0}\partial_1+\frac{A_2}{A_0}\partial_2)v(t,x)\nonumber\\
&&\qquad\qquad\quad +\frac{1}{A_0}\int_0^t (\frac{A_1}{A_0}\partial_1+\frac{A_2}{A_0}\partial_2)^2 v\big(s,x_1+\frac{A_1}{A_0}(s-t),x_2+\frac{A_2}{A_0}(s-t) \big) ds  \nonumber\\
&&\qquad\qquad\quad +(\frac{A_1}{A_0}\partial_1+\frac{A_2}{A_0}\partial_2)^2\chi\big(x_1-\frac{A_1}{A_0}t,x_2-\frac{A_2}{A_0}t \big), \nonumber\\
&&\partial_1 u(t,x)=\frac{1}{A_0}\int_0^t \partial_1 v\big(s,x_1+\frac{A_1}{A_0}(s-t),x_2+\frac{A_2}{A_0}(s-t) \big)ds  \nonumber\\
&&\qquad\qquad\quad +\partial_1\chi\big(x_1-\frac{A_1}{A_0}t,x_2-\frac{A_2}{A_0}t \big), \nonumber\\
&&\partial_2 u(t,x)=\frac{1}{A_0}\int_0^t \partial_2 v\big(s,x_1+\frac{A_1}{A_0}(s-t),x_2+\frac{A_2}{A_0}(s-t) \big)ds  \nonumber\\
&&\qquad\qquad\quad +\partial_2\chi\big(x_1-\frac{A_1}{A_0}t,x_2-\frac{A_2}{A_0}t \big). \nonumber
\end{eqnarray}
Combing the expressions in the above, we deduce that
\begin{equation}
v=A_i\partial_i u . \nonumber
\end{equation}
Inserting them into \eqref{quasi-equation}, we get
\begin{equation}\label{l2}
A_i\partial_i( \Box u-A_\lambda A_\nu N_{\mu\delta}\partial^2_{\lambda\mu} u \partial^2_{\nu\delta} u)=0.
\end{equation}
Let
\begin{equation}
(\Box u-A_\lambda A_\nu N_{\mu\delta}\partial^2_{\lambda\mu} u \partial^2_{\nu\delta} u )|_{t=0}=0, \nonumber
\end{equation}
then the expression inside the bracket of \eqref{l2} always vanishes.
Thus the quasilinear wave equations \eqref{quasi-equation} will be transformed into the following fully nonlinear wave equation
\begin{equation}\nonumber
\begin{cases}
\Box u=A_\lambda A_\nu N_{\mu\delta}\partial^2_{\lambda\mu} u \partial^2_{\nu\delta} u, \\
u(0,x)=\chi,\partial_t u(0,x) =\frac{1}{A_0}v_0-\frac{A_1}{A_0}\partial_1\chi-\frac{A_2}{A_0}\partial_2\chi,
\end{cases}
\end{equation}
where $\chi$ is a function satisfying the following relation
\begin{eqnarray}\nonumber
\begin{cases}
(\frac{A_1}{A_0}\partial_1 +\frac{A_2}{A_0}\partial_2)^2 \chi-\Delta \chi \\
\quad=-\frac{1}{A_0}\partial_t v(0)+\frac{A_1}{A^2_0}\partial_1 v(0)+\frac{A_2}{A^2_0}\partial_2 v(0)
+N_{\mu\delta}\partial_\mu v(0) \partial_\delta v(0),\\
(\chi(x),\nabla\chi(x))\in H^k_\Lambda.
\end{cases}
\end{eqnarray}
This gives \eqref{InitialData2} and \eqref{InitialDataCondition3}.

 In the above two cases, the quasilinear wave equations \eqref{quasi-equation} are both transformed to the fully nonlinear wave equations \eqref{eq-fully-transform}. Note the null condition assumption \eqref{quasi-null3}, then if the initial data $(\varphi,\psi)\in H^k_\Lambda$ with $k\geq 8$ and they satisfy \eqref{thm-InitialDataCondition},
we obtain the global existence result to \eqref{eq-fully-transform} according to Theorem \ref{thm} and Lemma \ref{lemmaSymmetry}. Consequently, the quasilinear wave equations \eqref{quasi-equation} have global classical solutions
since $v=A_i\partial_i u$.
\end{proof}

\begin{remark}
The transformation from quasilinear wave equations \eqref{quasi-equation} to fully nonlinear equation \eqref{eq-fully-transform} is reversible. Consider
\begin{equation} \label{wave-fully-2}
\begin{cases}
\Box u=A_\lambda A_\nu N_{\mu\delta}\partial^2_{\lambda\mu} u \partial^2_{\nu\delta} u, \\
u(0,x)=u_0(x),
\partial_t u(0,x)=u_1(x).
\end{cases}
\end{equation}
 We assume  \eqref{quasi-null} holds, let $(u_0,u_1)\in H^k_\Lambda$ and they satisfy the condition \eqref{thm-InitialDataCondition}.
According to Theorem \ref{thm} and Lemma \ref{lemmaSymmetry}, we obtain the global existence result for \eqref{wave-fully-2}.
Let $v=A_i\partial_i u$, then $v$ exits globally in time and satisfies the following quasilinear wave equation
\begin{equation} \label{A-class-quasi}
\begin{cases}
\Box v=A_l\partial_l(N_{\mu\delta}\partial_\mu v\partial_\delta v), \\
v(0,x)=A_i\partial_i u(0,x),\partial_t v(0,x) =A_i\partial_i\partial_t u(0,x).
\end{cases}
\end{equation}
\end{remark}

\section{Proof of Remark \ref{RemarkUniformBound}}
In this section, we are going to prove Remark \ref{RemarkUniformBound}.
\begin{proof}
Following the argument in the above section, let
\begin{equation*}
u(t,x)=\int_0^t v(\tau,x) d \tau +(-\Delta)^{-1}(|v_1|^2-|\nabla v_0|^2-v_1).
\end{equation*}
Then the equation \eqref{quasi-equation-standard} can be transformed to the following fully nonlinear wave equation
\begin{equation} \label{wave-fully}
\begin{cases}
\Box u= |\partial^2_t u|^2 - |\partial_t\nabla u|^2 , \\
u(0,\cdot)=(-\Delta)^{-1}(|v_1|^2-|\nabla v_0|^2-v_1),  \partial_t u(0,\cdot)=v_0.
\end{cases}
\end{equation}
For the nonlinearities in \eqref{wave-fully}, it's easy to see that
 they satisfy the null condition \eqref{null-condition}.
 Hence by Lemma \ref{LemmGammaEq}, we have
\begin{equation} \label{fully-wave-Gamma}
\Box \Gamma^a  u=\sum_{b+c+d = a}  N_d(\Gamma^b u,\Gamma^c u).
\end{equation}
Now we preform the highest order energy estimate for \eqref{wave-fully}.
Since the techniques are essentially the same as
the ones we have used in Section 4, we only sketch
the main line of the argument.

Let $k\geq 8$, $|a |\leq k-1$, multiplying the equation \eqref{fully-wave-Gamma} with
$\partial_t \Gamma^a  u $ and integrating over $\mathbb{R}^2$, we have
\begin{eqnarray*}
&&\frac{d}{dt} \int_{\mathbb{R}^2}
 {\frac12 (|\partial_t\Gamma^a  u|^2+|\nabla\Gamma^a  u|^2 )} dx\\
&&=\int_{\mathbb{R}^2} 2(\partial_t^2 \Gamma^a u \partial_t^2 u  - \partial_t \nabla\Gamma^a u \cdot \partial_t \nabla u)\partial_t \Gamma^a  u dx \\
&&\quad+\sum_{\tiny\begin{matrix}b+c+d = a \\ d\neq 0 \end{matrix}\tiny}
 \int_{\mathbb{R}^2} N_d(\Gamma^b u,\Gamma^c u) \partial_t \Gamma^a  u dx .
\end{eqnarray*}
For the highest order terms, we deduce from integration by parts that
\begin{eqnarray*}
&&2\int_{\mathbb{R}^2} (\partial_t^2 \Gamma^a u \partial_t^2 u  - \partial_t \nabla\Gamma^a u\partial_t \nabla u)\partial_t \Gamma^a  u dx \\
&&=-\int_{\mathbb{R}^2} |\partial_t^2\Gamma^a u|^2 \partial_t\Box u
 +\partial_t\int_{\mathbb{R}^2} |\partial_t \Gamma^a u|^2 \partial_t^2 u dx\\
&&\leq E_k(t) \|\partial_t\Box u \|_{L^\infty}+\partial_t\int_{\mathbb{R}^2} |\partial_t \Gamma^a u|^2 \partial_t^2 u dx.
\end{eqnarray*}
By Lemma \ref{lemK-S}, Lemma \ref{lemGamma}, and Lemma \ref{lemK-S-deriv}, we have
\begin{equation*}
\|\partial_t\Box u \|_{L^\infty}\lesssim \langle t\rangle^{-\frac 32} E_k^{\frac 12}(t).
\end{equation*}
The last term $\partial_t\int_{\mathbb{R}^2} |\partial_t \Gamma^a u|^2 \partial_t^2 u dx$ can be treated similar to \eqref{Estimate-purturbation1}-\eqref{Estimate-purturbation2} as a lower order perturbation.
On the other hand, following the paradigm of the energy estimates we have done in Section 4, we can estimate the remaining lower order terms as follows:
\begin{eqnarray*}
\sum_{\tiny\begin{matrix}b+c+d = a \\ d\neq 0 \end{matrix}\tiny}
 \int_{\mathbb{R}^2} N_d(\Gamma^b u,\Gamma^c u) \partial_t \Gamma^a  u dx
 \lesssim \langle t\rangle^{-\frac 32} E_k^{\frac 32}(t).
\end{eqnarray*}
Thus we infer by gathering the the above argument that
\begin{equation*}
\frac{d}{dt}E_{k}(t)\lesssim \langle t\rangle^{-\frac 32} E_k^{\frac 32}(t).
\end{equation*}
By continuity method, if $E_k(0)<\epsilon$ for a sufficiently small $\epsilon>0$, then the fully nonlinear wave equation \eqref{wave-fully} is globally well-posedness and the highest order energy is uniformly bounded : $E_k(t)<C\epsilon$ for some universal constant $C$. Note that we can transform \eqref{quasi-equation-standard} to \eqref{wave-fully}, hence
the highest energy of \eqref{quasi-equation-standard} is also uniformly bounded.
\end{proof}

\section{Acknowledgements}
The first two authors were in part supported by NSFC (Grant No. 11171072 and
11222107), NCET-12-0120, National Support Program for Young Top-Notch Talents, Shanghai
Shu Guang project and Shanghai Talent Development Fund. The third author was
supported by the NSF Grant DMS-1211806.

\end{document}